%% file: main.tex
\documentclass[12pt]{amsart}
\usepackage[utf8]{inputenc}

\usepackage{amsmath, amssymb, amsthm, mathrsfs, graphicx,float, mathtools}
\usepackage{comment}
\usepackage[shortlabels]{enumitem}
\usepackage[numbers]{natbib}
\usepackage[margin=1in]{geometry}
\usepackage{stmaryrd} 
\usepackage{xcolor}
\usepackage{tikz-cd}
\usepackage{tikz}
\usepackage[loadshadowlibrary]{todonotes}

\graphicspath{{graphics/}}

\input{macros.tex}

\title{Connectedness of the Moduli Stack of All Reduced Algebraic Curves}
\author{Sebastian Bozlee}
\date{June 2024}

\subjclass{14H10, secondary: 13E10}

\begin{document}

\begin{abstract}
Using moduli of equinormalized curves and Ishii's theory of territories, we prove that the moduli stack of all reduced $n$-pointed algebraic curves of fixed arithmetic genus is connected.
\end{abstract}

\maketitle

\section{Introduction}

Denote by $\UU_{g,n}$ the moduli stack of reduced but otherwise arbitrarily singular, connected, proper curves of genus $g$ with $n$ marked points. It is known to be representable by an algebraic stack, locally of finite type over $\Spec \ZZ$ with quasicompact and finite diagonal, see \cite[Theorem B.1]{smyth_zstable} and \cite[Tags 0D5A and 0E1G]{stacks-project}. We will call this the \emphbf{moduli stack of all curves}.

The moduli stack of all curves is not particularly ``nice:" it is known to possess many irreducible components, due to existence of non-smoothable curve singularities of various genera \cite{mumford_pathologies}, \cite[Theorem 1.11]{pinckham_thesis}. It is also known to be highly non-separated, as evidenced by the many alternate compactifications of the moduli stack of smooth pointed curves $\moduli_{g,n}$ \cite{schubert_pseudostable, hassett_weighted_curves, hassett_hyeon,fedorchuk_smyth}. In this article we show that it has at least one nice geometric property: it is connected.

\begin{theorem} \label{thm:connected}
The moduli stack of all reduced curves $\UU_{g,n}$ is connected and the fibers of $\UU_{g,n} \to \Spec \ZZ$ are geometrically connected for all $g,n \geq 0$.
\end{theorem}

Intuitively, our method is to construct a path in the moduli stack of all curves from any curve $X$ to a curve $X_{sm}$ with only smoothable singularities. Perhaps surprisingly, this can be done by varying the singularities of $X$ while leaving the normalization of $X$ fixed. We use in an essential way Ishii's theory of territories \cite{ishii_moduli_subrings} and its further development in \cite{bozlee_guevara_smyth}.

\section{Background: the theory of territories}

Fix an algebraically closed field $k$ for the remainder of the paper. In this section, we will relate reduced curve singularities to subalgebras of finite dimensional algebras, then we will study the moduli of such subalgebras. We begin with a lemma on subalgebras of $k^m$. A proof of part (i) is provided in \cite[Lemma 3.9]{bozlee_guevara_smyth}.

\begin{lemma} \label{lem:subalges_of_kn}
Let $m$ be a positive integer. Given a subset $S \subseteq \{ 1, \ldots, m \}$, write $1_S$ for the element of $k^m$ whose $i$th entry is 1 if $i \in S$ and 0 otherwise.
Then
\begin{enumerate}
  \item The $k$-subalgebras of $k^m$ are precisely the subsets of the form
\[
  B_{\mathcal{P}} = \Span \{ 1_P \mid P \in \mathcal{P} \}
\]
for some set partition $\mathcal{P}$ of $\{1, \ldots, m \}$.
  \item The only local subalgebra of $k^m$ is $k \cdot (1, \ldots, 1)$.
\end{enumerate} 
\end{lemma}

Let us recall the definition of the conductor ideal.

\begin{definition}
If $B$ is a subring of $A$, the \emphbf{conductor} $\Cond_{A/B}$ of $A$ over $B$ is defined by $\Cond_{A/B} = \Ann_B(A/B)$. It is not difficult to check $\Cond_{A/B}$ is the largest ideal of $B$ that is also an ideal of $A$.

If $\nu : \tilde{X} \to X$ is the normalization of a reduced curve $X$ over $k$, we set the \emphbf{conductor of $\tilde{X}$ over $X$} to be the quasicoherent $\OO_{\tilde{X}}$-module $\Cond_{\tilde{X}/X}$ determined by the rule that if $U = \Spec(A) \subseteq X$ is an affine open with preimage $V = \Spec(B) \subseteq \tilde{X}$, then $\Cond_{\tilde{X}/X}(V) = \Ann_{B}(A/B)$, considered as an $A$-module. More globally, we may write
\[
  \Cond_{\tilde{X}/X} = \nu^{-1}\AAnn_{\OO_X}(\nu_*{\OO_{\tilde{X}}}/\OO_X) \otimes_{\nu^{-1}\nu_*\OO_{\tilde{X}}} \OO_{\tilde{X}}.
\]
\end{definition}

With the definition of conductor in hand, we show that the complete local rings of reduced curve singularities are only certain subalgebras of products of power series rings.

\begin{lemma} \label{lem:curve_sings_to_power_series}
 Let $\widehat{\OO}$ be the complete local ring at a closed point of a reduced algebraic curve. Then $\widehat{\OO}$ is isomorphic to a $k$-subalgebra $B$ of $A = \prod_{i = 1}^m k\llbracket t_i \rrbracket$ for some $m$ such that
    \begin{enumerate}
        \item $\Cond_{A/B}$ is the $A$-ideal $(t_1^{c_1}, \ldots, t_m^{c_m})$ for some list of positive integers $c_1, \ldots, c_m$;
        \item $B \cap \left(\prod_{i=1}^m k \cdot 1\right) = k \cdot (1, \ldots, 1)$, i.e., the only tuples of constants contained in $B$ are those with equal coordinates.
    \end{enumerate}
    Moreover, the subalgebra $B$ depends only on a $k$-isomorphism of the normalization of $\widehat{\OO}$ with $A$. Conversely, any algebra $B$ satisfying these hypotheses is a complete local ring at a closed point of a reduced algebraic curve.
\end{lemma}

\begin{proof}
Let $\tilde{\OO}$ be the integral closure of $\widehat{\OO}$. Then, by the Cohen Structure Theorem, there is a $k$-isomorphism $\phi : \tilde{\OO} \to A = \prod_{i = 1}^m k\llbracket t_i \rrbracket$ for some integer $m$. Set $B = \phi(\widehat{\OO})$.
It remains to show (i) and (ii).

For (i), let $I = \Cond_{A/B}$. Since $A$ is a product, $I = I_1 \times \cdots \times I_m$ where $I_i$ is an ideal of the DVR $k\llbracket t_i \rrbracket$ for each $i = 1, \ldots, m$. By \cite[Chapter VIII, Proposition 1.16]{altman_kleiman}, $\dim_k(A/I)$ is finite, so it must be that $I_i = (t_i^{c_i})$ for some non-negative integer $c_i$ for each $i$. If $c_i$ were $0$ for any $i$, then $B$ would contain an idempotent other than 0 or 1, contradicting that $\Spec B$ is connected. We conclude that (i) holds.

For (ii), observe that $R = B \cap (\prod_{i = 1}^m k \cdot 1)$ is a $k$-subalgebra of $k^m$. Notice also that $R$ is isomorphic to $B / ((t_1, \ldots, t_m) \cap B)$, the quotient of a local ring by its maximal ideal, so $R$ is local. By Lemma \ref{lem:subalges_of_kn}, part (ii), $R = k \cdot (1, \ldots, 1)$.

For the converse, let $B$ be a subalgebra satisfying (i) and (ii). Let $\phi : \prod_{i=1}^m k[t_i] \to A$ be the natural inclusion. Consider the scheme $C = \Spec \phi^{-1}(B)$. It is one-dimensional since $\Spec \prod_{i=1}^m k[t_i] \to C$ is a finite map, reduced since $\phi^{-1}(B)$ is a subring of $\prod_{i=1}^m k[t_i]$, and finite type since $\phi^{-1}(B)$ is generated by elements of degree less than $\max \left\{2c_1, \ldots, 2c_m \right\}$. Altogether, $C$ is a reduced curve over $k$. Moreover, $B$ is its completion at the ideal $(t_1, \ldots, t_m) \cap C$. This must be the completion of $C$ at a maximal ideal since, by (ii),
\[
  C / ((t_1, \ldots, t_m) \cap C) \cong B / ((t_1, \ldots, t_m) \cap B) \cong B \cap \left(\prod_{i=1}^m k \cdot 1\right)
\]
is a field.
\end{proof}

Several important invariants of reduced curve singularities emerge in the statement of the lemma.

\begin{definition}
Let $x$ be a reduced curve singularity, and let $B$ be a corresponding subalgebra as in Lemma \ref{lem:curve_sings_to_power_series}.
The number $m$ in Lemma \ref{lem:curve_sings_to_power_series} is called the \emphbf{number of branches} of the singularity $x$. The integers $c_1, \ldots, c_m$ such that $\Cond_{A/B} = (t_1^{c_1}, \ldots, t_m^{c_m})$ are called the \emphbf{conductances} or \emphbf{branch conductances} of the singularity $x$; they are only well-defined up to a permutation of the branches. Their sum $c$ we call the \emphbf{total conductance} of the singularity.
\end{definition}

We've seen in Lemma \ref{lem:curve_sings_to_power_series} that only certain subrings of products of power series rings occur as the complete local rings of reduced curve singularities. The next lemma shows that by taking an additional quotient, those subrings with conductances bounded by $\vec{c}$ are in bijection with subalgebras of certain finite-dimensional algebras which we now introduce. 

\begin{definition}
Given a tuple $\vec{c} = (c_1,\ldots, c_m)$ of positive integers, let $A(\vec{c})$ denote the $k$-algebra
\[
  A(\vec{c}) = \prod_{i = 1}^m k[t_i]/(t_i^{c_i})
\]
and let $A^+(\vec{c})$ denote the subalgebra
\[
  A^+(\vec{c}) = \{ (f_1(t_1), \ldots, f_m(t_m)) \in A(\vec{c}) \mid f_i(0) = f_j(0) \text{ for all }i,j \}
\]
where ``constants are equal."
\end{definition}

The algebra $A(\vec{c})$ acts as an ``infinitesimal germ of the normalization" and $A^+(\vec{c})$ acts as an ``infinitesimal germ of the seminormalization" of the curve singularity. The $+$ notation is intended to evoke the appearance of a seminormalization, as the seminormalization of a curve singularity is a union of the coordinate axes in $\AA^n$.

\begin{lemma} \label{lem:power_series_to_finite}
  Fix a tuple $\vec{c} = (c_1, \ldots, c_m)$ of positive integers.

  \begin{enumerate}
  \item[(1)]
  The $k$-subalgebras $B$ of $A = \prod_{i = 1}^m k\llbracket t_i \rrbracket$ such that
  \begin{enumerate}
    \item[(i)] $B$ contains the $A$-ideal $(t_1^{c_1}, \ldots, t_m^{c_m})$ and
    \item[(ii)] $B$ has codimension $\delta$ in $A$ as a $k$-vector subspace
  \end{enumerate}
  are in natural bijection with $k$-subalgebras $R$ of $A(\vec{c})$ of codimension $\delta.$

  \item[(2)]
  The $k$-subalgebras $B$ of $A$ satisfying (i), (ii), and 
  \begin{enumerate}
    \item[(iii)] $B \cap \left(\prod_{i=1}^m k \cdot 1\right) = k \cdot (1, \ldots, 1)$
  \end{enumerate}
  are in natural bijection with $k$-subalgebras $R$ of $A^+(\vec{c})$ with codimension $g \coloneqq \delta - m + 1$.
  \end{enumerate}
\end{lemma}

\begin{proof}
Let $\pi : A \to A(\vec{c})$ be the quotient map. By the Fourth Isomorphism Theorem, the $k$-subalgebras $B$ of $\prod_{i = 1}^m k\llbracket t_i \rrbracket$ satisfying (i) are in bijection with the subalgebras $R$ of $A(\vec{c})$ via $B \mapsto \pi(B)$, $R \mapsto \pi^{-1}(R)$. It is clear that a subalgebra $B$ satisfies (iii) in addition to (i) if and only if $\pi(B)$ factors through $A^+(\vec{c})$.

It only remains to verify the claims about codimension. Suppose $B$ is a subalgebra satisfying (i), (ii).
By the Third Isomorphism Theorem, $\prod_{i=1}^m k\llbracket t_i \rrbracket / B$ and $A(\vec{c}) / \pi(B)$ are isomorphic $k$-modules, so $\pi(B)$ has codimension $\delta$ in $A(\vec{c})$. Observe that $A^+(\vec{c})$ has codimension $m - 1$ in $A(\vec{c})$. It follows that if $\pi(B)$ factors through $A^+(\vec{c})$, then $\pi(B)$ has codimension $\delta - (m - 1) = \delta - m + 1$ in $A^+(\vec{c})$.
\end{proof}

Again, important invariants of reduced curve singularities emerge in the lemma. We give them the following names.

\begin{definition}
If $x$ is a reduced curve singularity corresponding to $B$ via Lemma \ref{lem:curve_sings_to_power_series}, the integer $\delta = \dim_k (\prod_{i = 1}^m k\llbracket t_i \rrbracket / B)$ is called the \emphbf{delta invariant} of $x$ and $g = \delta - m + 1$ is called the \emphbf{genus} of $x$.
\end{definition}

\medskip

Altogether, Lemmas \ref{lem:curve_sings_to_power_series} and \ref{lem:power_series_to_finite} imply that reduced algebraic curve singularities of genus $g$ correspond to subalgebras of $A^+(\vec{c})$ of codimension $g$. The subalgebras of $A^+(\vec{c})$ of codimension $g$ are parametrized by moduli schemes called ``territories," originally introduced by Ishii for $k$-algebras in the context of Noetherian $k$-schemes. We start by defining an appropriate
moduli functor.

\begin{definition} (See \cite[Definition 1]{ishii_moduli_subrings}, \cite[Section 2]{bozlee_guevara_smyth})
Let $A$ be an $n$-dimensional $k$-algebra. Given an $k$-scheme $S$, a \emphbf{family of subalgebras of $A$ of corank $\delta$} on $S$ is a quasi-coherent $\OO_S$-subalgebra $\mathscr{B}$ of $f^*A$ such that the quotient $\OO_S$-module $f^*A / \mathscr{B}$ is locally free of rank $\delta$.

The \emph{$\delta$-territory of $A$} is the functor $\Ter^\delta_{A} : (\Sch/k)^{op} \to \Set$ by:
\begin{enumerate}
    \item If $S$ is a $k$-scheme, $\Ter^\delta_A(S)$ is the set of families of subalgebras of $A$ of corank $\delta$ on $S$.
    \item If $f : S \to S'$ is a morphism of $k$-schemes, then $\Ter^\delta_{A}(g) : \Ter^\delta_{A}(S') \to \Ter^\delta_{A}(S)$ is defined by taking a family of subalgebras to its pullback.
\end{enumerate}
\end{definition}

\begin{lemma} \label{lem:ter_representability} (See \cite[Theorem 1]{ishii_moduli_subrings}, \cite[Theorem 2.5]{bozlee_guevara_smyth})
Let $A$ be an $n$-dimensional $k$-algebra. The functor $\Ter^\delta_{A}$ is represented by a closed subscheme of $\Grass(n-\delta, A)$, which by common abuse of notation we also denote by $\Ter^\delta_A$. In particular, $\Ter^\delta_A$ is a projective $k$-scheme.
\end{lemma}

\begin{remark}
These notions are extended to finite locally free algebras over an arbitrary base in \cite{bozlee_guevara_smyth}.
\end{remark}

We will take $A = A(\vec{c})$ or $A = A^+(\vec{c})$ for the remainder of the paper. Lemmas \ref{lem:curve_sings_to_power_series} and \ref{lem:power_series_to_finite} imply that each isomorphism class of reduced curve singularity over $k$ is represented by a $k$-point of $\Ter^g_{A^+(\vec{c})}$, where $g$ is the genus of the singularity and $\vec{c} = (c_1, \ldots, c_m)$ is at least its vector of branch conductances. Because $A^+(\vec{c})$ is a codimension $m - 1$ subring of $A(\vec{c})$, such a curve singularity is also represented by a $k$-point of $\Ter^\delta_{A(\vec{c})}$, where $\delta$ is the delta invariant of the singularity. Accordingly, we will consistently use the letter $g$ with the territory of $A^+(\vec{c})$ and $\delta$ with the territory of $A(\vec{c})$.

\bigskip

We now consider a global analogue of a territory. Intuitively, it parametrizes ways of a changing a scheme $\tilde{X}$ to a scheme $X$ by crimping jets and gluing points within a finite closed subscheme $Z \subseteq \tilde{X}$. We will apply this in the case that $\tilde{X}$
is a family of smooth, possibly disconnected curves; then $X$ will vary over the curves with normalization $\tilde{X}$ such that the $V(\Cond_{\tilde{X}/X})$ is contained in $Z$.

\begin{definition} \label{def:global_territory}
Suppose that $\tilde{\pi} : \tilde{X} \to \Spec k$ is a projective scheme. Let $\iota : Z \to \tilde{X}$ be a closed subscheme finite over $\Spec k$. Write $\mathscr{I}_{Z/\tilde{X}}$ for the sheaf of ideals of $Z$ in $\tilde{X}$. Define a functor $\Ter^\delta_{Z/\tilde{X}}$ that assigns to a $k$-scheme $S$ the set of morphisms of $S$-schemes (modulo isomorphism)
\[
  \begin{tikzcd}
    \tilde{X}_S \ar[rr, "\nu"] \ar[dr, "\tilde{\pi}_S"'] &  & X \ar[dl, "\pi"] \\
     & S
  \end{tikzcd}
\]
such that
\begin{enumerate}
  \item $\nu$ is a finite morphism of schemes,
  \item the associated map of sheaves $\nu^\sharp : \OO_{X} \to \nu_*\OO_{\tilde{X}_S}$ is injective,
  \item $\nu^\sharp(\OO_{X})$ contains $\nu_*\mathscr{I}_{Z_S/X_S}$,
  \item letting $\Delta = (\nu_*\OO_{\tilde{X}_S}) / \nu^\sharp(\OO_{X})$, the $\OO_S$-module $\pi_*\Delta$
  is locally free of rank $\delta$.
\end{enumerate}
Restriction morphisms $\Ter^\delta_{Z/\tilde{X}}(S') \to \Ter^\delta_{Z/\tilde{X}}(S)$ are given by pullback.
\end{definition}

If $\tilde{X}$ is a smooth, possibly disconnected curve, and $X$ is a family of connected curves, then $\nu : \tilde{X}_S \to X$ is a family of ``equinormalized curves" over $S$. Moduli of equinormalized curves are studied in more detail in \cite{bozlee_guevara_smyth}.

\medskip

While we do not put explicit conditions on $X \to S$ in Definition \ref{def:global_territory}, it possesses several good properties, which we state in the following proposition.

\begin{proposition} \label{prop:X_to_S_is_nice} (See \cite[Section 4]{bozlee_guevara_smyth})
Let $\nu : \tilde{X}_S \to X$ be an $S$-morphism as in Definition \ref{def:global_territory}. Then:
\begin{enumerate}
  \item $\nu^{-1}(\nu(Z_S)) = Z_S$;
  \item $\nu$ restricts to an isomorphism $\tilde{X}_S - Z_S \to X - \nu(Z_S)$;
  \item the map $\pi : X \to S$ is proper, flat, and of finite presentation.
\end{enumerate}
\end{proposition}

\medskip

Although the data might seem more weighty, the diagrams in Definition \ref{def:global_territory} are also parametrized by the territory of a finite-dimensional algebra, namely $\tilde{\pi}_*\OO_Z$.
Let us sketch how, assuming $\nu^\sharp$ is an inclusion to declutter the notation. Let $\nu(Z_S)$ denote the scheme theoretic image of $Z_S$.
We construct a diagram:
\[
\begin{tikzcd}
     & 0 \ar[d] & 0 \ar[d] & \\
     0 \ar[r] & \nu_*\mathscr{I}_{Z_S/\tilde{X}_S} \ar[r, equals] \ar[d] & \nu_*\mathscr{I}_{Z_S/\tilde{X}_S} \ar[r] \ar[d] & 0 \ar[d] \\
     0 \ar[r] & \OO_X \ar[r] \ar[d] & \nu_*\OO_{\tilde{X}_S} \ar[r] \ar[d] & \Delta \ar[r] \ar[d, "\sim"] & 0 \\
     0 \ar[r] & \OO_{\nu(Z_S)} \ar[r] \ar[d] & \nu_*\OO_{Z_S} \ar[r] \ar[d] & \ol{\Delta} \ar[r] \ar[d] & 0 \\
     & 0 & 0 & 0
\end{tikzcd}
\]
The middle row is exact by definition. The middle column is exact since $\nu$ is finite. The first column is exact by construction of the scheme theoretic image and the fact that $\nu_*\mathscr{I}_{Z_S/\tilde{X}_S}$ factors through $\OO_X$. We construct the arrow $\OO_{\nu(Z_S)} \to \nu_*\OO_{Z_S}$ by the universal property of quotients and set $\ol{\Delta}$ to the cokernel. Again by the universal property of quotients, there is a map $\Delta \to \ol{\Delta}$ completing the diagram. The last column is exact by the snake lemma. Observe that $\nu(Z_S) \to S$ is finite, so $R^1\pi_*\OO_{\nu(Z_S)} = 0$. Therefore, applying $\pi_*$, we get a subalgebra $\pi_*\OO_{\nu(Z_S)}$ of
$\pi_*\nu_*\OO_{Z_S} \cong \tilde{\pi}_*\OO_Z \otimes \OO_S$, the quotient by which is $\pi_*\ol{\Delta} \cong \pi_*\Delta$, a locally free $\OO_S$-module of rank $\delta.$ That is, starting with $\tilde{X}_S \to X$ in $\Ter^\delta_{Z/\tilde{X}}(S)$, we get a point $\pi_*\OO_{\nu(Z_S)}$ of $\Ter^\delta_{\tilde{\pi}_*(\OO_Z)}(S)$. The following lemma asserts that this yields an isomorphism of functors.

\begin{lemma} \cite[Theorem 4.5]{bozlee_guevara_smyth} \label{lem:global_territory_to_local_territory}
With notation as in Definition \ref{def:global_territory}, there is an isomorphism of functors $\Ter^\delta_{Z/\tilde{X}} \to \Ter^{\delta}_{\tilde{\pi}_*\OO_Z}$, given by sending a morphism of $S$-schemes $\tilde{X}_S \to X \in \Ter^{\delta}_{Z/\tilde{X}}(S)$ as in Definition \ref{def:global_territory} to $\pi_*(\OO_{\nu(Z_S)})$. In particular, $\Ter^\delta_{Z/\tilde{X}}$ is representable by a projective scheme.
\end{lemma}

In the case that $\tilde{X}$ is a smooth, possibly disconnected curve over $k$ and $Z$ is a closed subscheme consisting of points $q_1, \ldots, q_m$ with respective multiplicities $c_1, \ldots, c_m$, then
\[
  \Ter^\delta_{Z/\tilde{X}} \cong \Ter^\delta_{A(\vec{c})}.
\]

\medskip

Crucially for our main result, Ishii has shown that when $A$ is a local $k$-algebra, any two $k$-points of the territory $\Ter^\delta_{A}$ are connected by a chain of $\AA^1$s \cite[Corollary 2]{ishii_moduli_subrings}. We immediately obtain the following lemma.

\begin{lemma} \label{lem:territory_connected} 
For any $g$ and $\vec{c}$, $\Ter^g_{A^+(\vec{c})}$ is connected.
\end{lemma}

By contrast, $\Ter^\delta_{A(\vec{c})}$ is rarely connected. It has connected components corresponding to different choices of which points of $\Spec A(\vec{c})$ are glued together, and with what genus they are glued together. In order to make this precise, we need notation for the points of $\Spec A(\vec{c})$ and the restriction of $\vec{c}$ to a subset of $\{ 1, \ldots, m \}$.

\begin{definition}
For each $i = 1, \ldots, m$, write $\mathfrak{q}_i$ for the point $V((1, \ldots, t_i, \ldots, 1))$ of $\Spec A(\vec{c})$.
\end{definition}

\begin{definition}
If $\vec{c} = (c_1, \ldots, c_m)$ is a tuple of positive integers and $P = \{i_1, \ldots, i_p\}$ is a subset of $\{ 1, \ldots, m \}$, with $i_1 < \cdots < i_p$, denote by $\vec{c}|_P$ the tuple $(c_{i_1}, \ldots, c_{i_p})$.
\end{definition}

\begin{lemma} \label{lem:subalges_of_kn_maps}
Resume the notation of Lemma \ref{lem:subalges_of_kn}.
\begin{enumerate}
  \item If $\mathcal{P}$ is a partition of $\{1, \ldots, m \},$ the parts of $\mathcal{P}$ are in bijection with the points of $\Spec B_{\mathcal{P}}$ via
  \[
     P \mapsto \mathfrak{q}_P.
  \]
  where $\mathfrak{q}_P$ is the prime ideal $1_{P^c} \cdot B_{\mathcal{P}}$ of $B_{\mathcal{P}}$.
  \item For each partition $\mathcal{P}$ of $\{1, \ldots, m\}$, the map
  \[
    \Spec k^m = \{ \mathfrak{q}_{\{1\}}, \ldots, \mathfrak{q}_{\{m\}} \} \to \Spec B_{\mathcal{P}} = \{ \mathfrak{q}_P \mid P \in \mathcal{P} \}
  \]
  induced by the inclusion $B_{\mathcal{P}} \to k^m$ sends $\mathfrak{q}_{\{i\}}$ to $\mathfrak{q}_P$ where $P$ is the part of the partition containing $i$.
\end{enumerate}
\end{lemma}
\begin{proof}
Omitted.
\end{proof}

For any $\delta$ and any tuple $\vec{c}$, $\Ter^\delta_{A(\vec{c})}$ is very nearly a disjoint union of products of territories of $A^+(\vec{c})$.

\begin{proposition} (\cite[Theorem 3.10]{bozlee_guevara_smyth}) \label{prop:ter_as_union_of_ter_of_seminormalizations}
There is an infinitesimal thickening
\[
  \coprod_{\mathcal{P}, g} \left( \prod_{P \in \mathcal{P}} \Ter^{g(P)}_{A^+(\vec{c}|_P)} \right) \to \Ter^\delta_{A(\vec{c})}
\]
where the disjoint union varies over the partitions $\mathcal{P}$ of $\{1, \ldots, m\}$ together with functions $g : \mathcal{P} \to \NN$ such that $\sum_{P \in \mathcal{P}} (g(P) + |P| - 1) = \delta$.

Moreover, if $B$ is any $k$-point of the summand indexed by $(\mathcal{P},g)$, then the induced map $\Spec A(\vec{c}) \to \Spec B$ identifies points $\mathfrak{q}_i$ and $\mathfrak{q}_j$ if and only if $i, j$ belong to
 the same part of the partition $\mathcal{P}.$
\end{proposition}

The map is given by taking a tuple $(B_P)_{P \in \mathcal{P}} \in \prod_{P \in \mathcal{P}} \Ter^{g(P)}_{A^+(\vec{c}|_P)}$ to $\prod_{P \in \mathcal{P}} B_P$, modulo renumbering $t_i$'s.


\medskip

Let us summarize for $\tilde{X}$ a smooth, possibly disconnected, curve and $Z$ a closed subscheme consisting of points $q_1, \ldots, q_m$ with respective multiplicities $c_1, \ldots, c_m$.
Suppose $\nu : \tilde{X} \to X$ is a $k$-point of $\Ter^\delta_{Z/\tilde{X}}$. Then $\nu$ is the normalization of $X$. The infinitesimal thickening
\[
  \coprod_{\mathcal{P}, g} \left( \prod_{P \in \mathcal{P}} \Ter^{g(P)}_{A^+(\vec{c}|_P)} \right) \to \Ter^\delta_{A(\vec{c})} \cong \Ter^\delta_{Z/\tilde{X}}
\]
is a bijection on $k$-points. 
Let $\mathcal{P}$ be the partition and $g : \mathcal{P} \to \ZZ_{\geq 0}$ the genus function indexing the summand corresponding to $\nu$, and let $(x_P)_{P \in \mathcal{P}} \in \prod_{P \in \mathcal{P}} \Ter^{g(P)}_{A^+(\vec{c}|_P)}$ be the point of the summand corresponding to $\nu$. Then by Proposition \ref{prop:ter_as_union_of_ter_of_seminormalizations}, $\nu$ identifies points $q_i$ and $q_j$ where $i$ and $j$ belong to the same part of the partition $\mathcal{P}$, leaving $X$ with a singularity for each part of $\mathcal{P}$. Moreover, the singularity corresponding to a part $P \in \mathcal{P}$ has genus $g(P)$. Its local ring can be found by taking the $P$-coordinate $x_P$ and applying Lemmas \ref{lem:power_series_to_finite} and \ref{lem:curve_sings_to_power_series}.

\section{Connectedness}

\begin{definition}
Let $n_1, \ldots, n_m$ positive integers. Let $X_{n_1,\ldots,n_m}$ denote the singularity $\Spec k \oplus t_1^{n_1}k\llbracket t_1 \rrbracket \oplus \cdots \oplus t_m^{n_m}k\llbracket t_m \rrbracket$.
\end{definition}

The singularities $X_{n_1,\ldots,n_m}$ are known variously as universal singularities or partition singu\-larities \cite{stevens_versal_deformation}. We remark that $X_{n_1,\ldots, n_m}$ has $m$ branches, branch conductances $n_1$,$\ldots$, $n_m$, genus $n_1 + \cdots + n_m - m$, and delta invariant $n_1 + \cdots + n_m - 1$. We also remark that $X_{n_1, \ldots, n_m}$ is the transverse union of the unibranch singularities $X_{n_i} = \Spec k[t^i : i \geq n_i ]$ for $i = 1,\ldots,m.$

\begin{lemma} \label{lem:unibranch_partition_curves_smoothable}
For any $n_1,\ldots,n_m$, the singularity $X_{n_1,\ldots,n_m}$ is smoothable. 
\end{lemma}
\begin{proof}
It suffices to show that $X_{n_1,\ldots,n_m}$ is a limit of smoothable singularities.
Let $R = k[x_{i,j} : i = 1,\ldots, m, j = 1, \ldots, n_i]$ and consider the subring
\[
  B = R\cdot 1 \oplus \bigoplus_{i=1}^m \langle (t_i - x_{i,1}) \cdots (t_i - x_{i,n_i}) \rangle \quad \text{of} \quad A = \prod_{i=1}^m R[t_i],
\]
where $\langle (t_i - x_{i,1}) \cdots (t_i - x_{i,n_i}) \rangle$ denotes the ideal generated by $(t_i - x_{i,1}) \cdots (t_i - x_{i,n_i})$ in $R[t_i]$. The quotient $A/B$ is a finite free $R$-module of rank $n_1 + \cdots + n_m - 1$. (A basis
is given by the images of $x_i, \ldots, x_i^{n_i - 1}$ and $(1,0,\ldots, 0), \ldots, (0, \ldots, 0, 1, 0)$.) It follows $\Spec B \to \Spec R$ is a flat family of affine curves. The generic fiber of $\Spec B \to \Spec R$ has one singularity obtained by gluing the points $t_i = x_{i,j}$ to a single point transversally. This is an ordinary $n_1 + \cdots + n_m$-fold point (i.e., the same singularity as that formed by the coordinate axes in $\AA^{n_1 + \cdots + n_m}$.) Ordinary $n_1 + \cdots +n_m$-fold points are smoothable singularities. It follows that the singularity $X_{n_1,\ldots,n_m}$, which appears in the fiber over the point where all $x_{i,j} = 0$, is a limit of smoothable singularities, as desired.
\end{proof}

\begin{corollary} \label{cor:there_is_a_smoothable_singularity}
Each non-empty $\Ter^g_{A^+(\vec{c})}$ has a $k$-point corresponding to a smoothable singularity.
\end{corollary}
\begin{proof}
In order that $\Ter^g_{A^+(\vec{c})}$ be non-empty, it is necessary that $\delta + 1 \leq c$ \cite[Chapter VIII, Proposition 1.16]{altman_kleiman}.
Suppose $g$ as a non-negative integer and $\vec{c} = (c_1,\ldots,c_m)$ is a vector of positive integers with sum $c$ such that $\delta + 1 = g + m \leq c$. It follows $g \leq \sum_{i=1}^m (c_i - 1)$. As the latter is a sum of non-negative integers, we can find integers $n_1,\ldots, n_m$ with $1 \leq n_i \leq c_i$ for all $i$ such that $g = \sum_{i=1}^m (n_i - 1)$. Then $X_{n_1,\ldots,n_m}$ is a smoothable, $m$-branch singularity with genus $g$, and its ring of regular functions $k \oplus t_1^{n_1}k\llbracket t_1 \rrbracket \oplus \cdots \oplus t_m^{n_m}k\llbracket t_m \rrbracket$ contains the ideal $(t_1^{c_1}, \ldots, t_m^{c_m})$. Quotienting by this ideal yields a subalgebra of $A^+(\vec{c})$ corresponding to $X_{n_1,\ldots,n_m}$. This is the required point of $\Ter^g_{A^+(\vec{c})}.$
\end{proof}

\bigskip

We now prove Theorem \ref{thm:connected}.

\begin{proof}
Let $x = (X, p_1,\ldots,p_n) \in \mathcal{U}_{g,n}(k)$ be an arbitrary reduced, connected, proper $n$-pointed curve over $k$. Let $\nu : \tilde{X} \to X$ be the normalization of $X$, and let $\tilde{p}_1,\ldots, \tilde{p}_n \in \tilde{X}$ be lifts of $p_1, \ldots, p_n$. Let $Z = V(\Cond_{\tilde{X}/X}) \subseteq \tilde{X}$ be the vanishing of the conductor ideal, and let $q_1,\ldots,q_m$ be the points of its support with respective multiplicities $c_1, \ldots, c_m$. Observe that $\tilde{\pi}_*\OO_Z \cong A(\vec{c})$. Let $\delta = \dim_k(\nu_*\OO_{\tilde{X}}/\OO_X).$

By Lemma \ref{lem:global_territory_to_local_territory}, the tuple $(\nu : \tilde{X} \to X, \tilde{p}_1,\ldots,\tilde{p}_n)$ determines a $k$-point $\tilde{x}$ of $\Ter^\delta_{Z/\tilde{X}} \cong \Ter^\delta_{A(\vec{c})}$. By Proposition \ref{prop:ter_as_union_of_ter_of_seminormalizations}, $\tilde{x}$ factors through an open and closed subscheme $T$ isomorphic to $\prod_{P \in \mathcal{P}} \Ter^{g(P)}_{A^+(\vec{c}|_{P})}$ for some partition $\mathcal{P}$ of
$\{1,\ldots, m\}$ and function $g : \mathcal{P} \to \NN$. We know by Lemma \ref{lem:territory_connected} that $T$ is connected.

Write $\nu_T : \tilde{X}_T \to X_T$ for the family of equinormalized curves obtained by restricting the universal element of $\Ter^\delta_{Z/\tilde{X}}$ to $T$. The family $\pi_T : X_T \to T$ is flat, proper, and of finite presentation by Proposition \ref{prop:X_to_S_is_nice}. Since $X_T|_{\tilde{x}} = X$ is connected and $T$ is connected, all geometric fibers of $\pi_T$ are connected. Therefore $\pi_T : X_T \to T$ determines a morphism $\phi : T \to \mathscr{U}_{g,n} \times \Spec k$.

The image $\phi(T)$ is connected, since $T$ is connected, and contains $x = \phi(\tilde{x})$ by construction. Since $\moduli_{g,n} \times \Spec k$ is irreducible \cite{deligne_mumford}, its closure in $\UU_{g,n} \times \Spec k$ is connected. By Corollary \ref{cor:there_is_a_smoothable_singularity}, $\phi(T)$ contains a point $y$ corresponding to a curve with only smoothable singularities. Since a curve is smoothable if and only if it has smoothable singularities \cite[Corollary 29.10]{hartshorne_deformation_theory}, $y$ is in the closure of $\moduli_{g,n} \times \Spec k$. Then $\phi(T)$ and the closure of $\moduli_{g,n} \times \Spec k$ have the point $y$ in common, so the union of $\phi(T)$ and the closure of $\moduli_{g,n} \times \Spec k$ is connected. Since $x$ was arbitrary, every point of $\UU_{g,n} \times \Spec k$ belongs to the same connected component as $\moduli_{g,n} \times \Spec k$. As each connected component of an algebraic stack locally of finite presentation over $k$ possesses a $k$ point, it follows that the fibers of $\UU_{g,n} \to \Spec \ZZ$ are geometrically connected. Since $\moduli_{g,n}$ is connected and intersects each fiber, $\UU_{g,n}$ itself is also connected.
\end{proof}

\bibliographystyle{amsalpha}
\bibliography{bibliography}

\end{document}

%% file: macros.tex
\newtheorem{theorem}{Theorem}
\newtheorem{proposition}[theorem]{Proposition}
\newtheorem{lemma}[theorem]{Lemma}
\newtheorem{corollary}[theorem]{Corollary}

\theoremstyle{definition}
\newtheorem{definition}[theorem]{Definition}

\newtheorem{remark}[theorem]{Remark}

\numberwithin{theorem}{section}

\makeatletter
\providecommand{\leftsquigarrow}{%
  \mathrel{\mathpalette\reflect@squig\relax}%
}
\newcommand{\reflect@squig}[2]{%
  \reflectbox{$\m@th#1\rightsquigarrow$}%
}
\makeatother

\renewcommand{\phi}{\varphi} 
\renewcommand{\vec}{\mathbf} 

\newcommand{\ol}[1]{\overline{#1}}
\newcommand{\emphbf}[1]{\textit{#1}}

\definecolor{sebgreen1}{rgb}{0.019,0.317,0.149}
\definecolor{sebgreen2}{rgb}{0.784,0.952,0.780}

\definecolor{chrisyellow1}{rgb}{1,0.73,0}
\definecolor{chrisyellow2}{rgb}{0.99,0.99,0.75}

\definecolor{sebdarkgreen}{rgb}{0.019,0.317,0.149}
\definecolor{sebgreen}{RGB}{36,157,55}
\definecolor{seblightgreen}{rgb}{0.784,0.952,0.780}
\definecolor{sebblue}{RGB}{0,123,194}
\definecolor{seblightblue}{RGB}{194,242,255}

\usepackage[pagebackref]{hyperref}
\hypersetup{
  colorlinks   = true,          
  urlcolor     = sebblue,          
  linkcolor    = purple,          
  citecolor   = sebblue             
}

\renewcommand{\AA}{\mathbb{A}}

\newcommand{\NN}{\mathbb{N}}

\newcommand{\ZZ}{\mathbb{Z}}

\newcommand{\OO}{\mathscr{O}}

\newcommand{\UU}{\mathcal{U}}

\newcommand{\Ann}{\mathrm{Ann}}
\newcommand{\AAnn}{\underline{\mathrm{Ann}}}

\newcommand{\Cond}{\mathrm{Cond}}

\newcommand{\Grass}{\mathrm{Gr}}

\newcommand{\moduli}{\mathcal{M}}

\newcommand{\Span}{\mathrm{Span}\,}
\newcommand{\Spec}{\mathrm{Spec}\,}

\renewcommand{\tilde}{\widetilde}

\newcommand{\Ter}{\mathrm{Ter}}

\newcommand{\Sch}{\mathrm{Sch}}
\newcommand{\Set}{\mathrm{Set}}

\theoremstyle{definition}

%% file: bibliography.bib
@article{deligne_mumford,
  author = {Deligne, Pierre and Mumford, David},
  title = {The irreducibility of the space of curves of given genus},
  journal = {Publications Math\'ematiques de l'IH\'ES},
  volume = 36,
  issue = 1,
  pages = {75-109},
  year = 1969,
  month = {January}
}

@article{hassett_weighted_curves,
  author = {Hassett, Brendon},
  title = {Moduli spaces of weighted pointed stable curves},
  journal ={Advances in Mathematics},
  volume = {173},
  year = {2003},
  pages = {316-352}
  }

@article{hassett_hyeon,
  author = {Hassett, Brendon and Hyeon, Donghoon},
  title = {Log canonical models for the moduli space of curves: first divisorial contraction},
  journal = {Transactions of the American Mathematical Society},
  volume = 361,
  number = 8,
  month = {August},
  year = {2009},
  pages = {4471-4489}
}

@article{ishii_moduli_subrings,
  author = {Ishii, Shihoko},
  title = {Moduli of Subrings of a Local Ring},
  journal = {Journal of Algebra},
  volume = 67,
  year = 1980,
  pages = {504-516}
}

@article{schubert_pseudostable,
  author={Schubert, David},
  title={A new compactification of the moduli space of curves},
  journal={Compositio Mathematica},
  volume={78},
  number={3},
  pages={297-313},
  year={1991}
}

@article{smyth_zstable,
   author = {Smyth, David},
   title = {Towards a classification of modular compactifications of $\mathscr{M}_{g,n}$},
   journal = {Inventiones Mathematicae},
   volume = {192},
   number = {2},
   year = {2013},
   month={May},
   pages={459-503}
}

@book{altman_kleiman,
    AUTHOR = {Altman, Allen and Kleiman, Steven},
     TITLE = {Introduction to {G}rothendieck duality theory},
    SERIES = {Lecture Notes in Mathematics, Vol. 146},
 PUBLISHER = {Springer-Verlag, Berlin-New York},
      YEAR = {1970},
     PAGES = {ii+185},
   MRCLASS = {14.55},
  MRNUMBER = {0274461},
MRREVIEWER = {Y. Nakai},
}

@misc{stacks-project,
    shorthand    = {Stacks},
    author       = {The {Stacks Project Authors}},
    title        = {\textit{Stacks Project}},
    howpublished = {\url{https://stacks.math.columbia.edu}},
    year         = {2019}
}

@article{stevens_versal_deformation,
  author = {Stevens, Jan},
  title = {The versal deformation of universal curve singularities},
  journal = {Abh. Math. Semin. Univ. Hambg.},
  volume = 63,
  pages = {197--213},
  year = {1993},
  url = {{https://doi.org/10.1007/BF02941342}},
  doi = {10.1007/BF02941342}
}

@article{fedorchuk_smyth,
    author = {Maksym Fedorchuk and David Ishii Smyth},
    title = {Alternate compactifications of moduli spaces of curves},
    year = {2010},
    journal = {Handbook of Moduli},
    volume = 24,
    pages = {331--414},
    publisher = {International Press},
}

@article{bozlee_guevara_smyth,
  author = {Bozlee, Sebastian and Guevara, Christopher and Smyth, David Ishii},
  title = {A stratification of moduli of arbitrarily singular curves},
  journal = {arXiv e-prints},
  year={2024},
  pages = {arXiv:2307.10013},
  eprint={2307.10013},
  archivePrefix={arXiv},
  primaryClass={math.AG},
  url={https://arxiv.org/abs/2307.10013},
  doi = {10.48550/arXiv.2307.10013}
}

@article{mumford_pathologies,
  author = {Mumford, David},
  title = {{Pathologies IV}},
  year = {1975},
  journal = {American Journal of Mathematics},
  volume = {97},
  number = {3},
  pages = {847--849},
  url = {{http://dx.doi.org/10.2307/2373780}},
  doi = {10.2307/2373780}
}

@article{pinckham_thesis,
  author = {Pinkham, Henry},
  title = {Deformations of algebraic varieties with $\mathbb{G}_m$ action},
  journal = {Ast\'erisque},
  number = {20},
  year = {1974},
  pages = {1--140},
  url = {{http://numdam.org/item/AST_1974__20__1_0/}}
}

@book{hartshorne_deformation_theory,
  author = {Hartshorne, Robin},
  title = {Deformation Theory},
  publisher = {Springer},
  address = {New York, NY},
  year = {2010},
  doi = {10.1007/978-1-4419-1596-2},
  url = {https://doi.org/10.1007/978-1-4419-1596-2}
}
